\documentclass{amsart} 

\usepackage{amssymb}%
\usepackage{graphicx}

\usepackage{hyperref}

\makeatletter

  \newcommand{\labelformat}[1]{%
    \expandafter\renewcommand\csname p@#1\endcsname[1]%
  }

  \newcommand{\numberlike}[2]{%
     \expandafter\def\csname c@#1\endcsname{%
         \expandafter\csname c@#2\endcsname}%
  }
\makeatother
  
  \newcommand{\mynewtheorem}[2]{
    \newtheorem{#1}{#2}[section]
    \labelformat{#1}{#2~##1}
  }

  \theoremstyle{plain}
  \mynewtheorem{lemma}{Lemma}
  \mynewtheorem{theorem}{Theorem}
   \numberlike{theorem}{lemma}    
  \mynewtheorem{corollary}{Corollary}
   \numberlike{corollary}{lemma}
  \mynewtheorem{proposition}{Proposition}
   \numberlike{proposition}{lemma}

  \theoremstyle{definition}
  \mynewtheorem{definition}{Definition}
   \numberlike{definition}{lemma}

  \theoremstyle{remark}
  \mynewtheorem{Remark}{Remark}
   \numberlike{Remark}{lemma}   
  \mynewtheorem{example}{Example}
   \numberlike{example}{lemma}

\newcommand{\NN}{\mathbb N}

\newcommand\set[1]{\left\{\,#1\,\right\}}  
\newcommand\with{\ \vrule\ }  
\newcommand{\ohne}[1]{\setminus\{#1\}}
\newcommand{\cc}[1]{#1^c}

\newcommand{\lk}[2][\de]{lk_{#1}(#2)}  
\newcommand{\bd}[1]{\partial #1} 
\newcommand{\join}{\mathrel{\ast}}
\DeclareMathOperator{\sus}{\Sigma_{1,v}}
\DeclareMathOperator{\tsus}{\Sigma_{2}}

\newcommand{\de}{\ensuremath{\Delta} }
\newcommand{\al}[1][\de]{\ensuremath{\alpha(#1)} }
\newcommand{\defa}{:=} 
\newcommand{\MM}[1][\de]{\ensuremath{\mathcal M(#1)} }
\newcommand{\LL}[1][\de]{L_{#1}}
\newcommand{\hh}[2]{b_{#1}(#2)}
\newcommand{\I}[2]{\ensuremath{I(#1, #2)}}

\begin{document}
\title[On homology spheres with few minimal non-faces]{On homology spheres with few minimal non-faces}

\author{Lukas Katth\"an}
\address{Fachbereich Mathematik und Informatik, Phillips Universität, 35032 Marburg, Germany}%

\email[L. Katth\"an]{katthaen@mathematik.uni-marburg.de}%

\thanks{This work was partially supported by the DAAD and the DFG}
\date{\today}


\begin{abstract}
Let \de be a $(d-1)$-dimensional homology sphere on $n$ vertices with $m$ minimal non-faces. We consider the invariant $\al = m - (n-d)$ and prove that for a given value of $\alpha$, there are only finitely many homology spheres that cannot be obtained through one-point suspension and suspension from another. Moreover, we describe all homology spheres with $\al$ up to four and, as a corollary, all homology spheres with up to eight minimal non-faces. To prove these results we consider the lcm-lattice and the nerve of the minimal non-faces of $\de$. Also, we give a short classification of all homology spheres with $n-d \leq 3$.
\end{abstract}

\maketitle

\section{Introduction}
Let \de be a simplicial complex with vertices in $[n] \defa \set{1,\hdots,n}$ for an $n \in \mathbb N$. A minimal non-face of \de is a set $M \subset [n]$ with $M \notin \de$, but $M\ohne{i}\in \de$ for every $i \in M$. Note that a face of \de may be described as a set not containing a minimal non-face, hence \de is completely determined by its set of minimal non-faces. In the present work we consider the structure of minimal non-faces of homology spheres. Using a result about the lcm-lattice of a monomial ideal by Gasharov, Peeva and Welker \cite{lcm}, we obtain structural results and descriptions of homology spheres with few minimal non-faces.
This paper is organized as follows: In Section \ref{sec:elem} we introduce some notation and constructions for general simplicial complexes. In Section \ref{sec:lcm} we specialize to homology spheres and prove that for a given value of $\al \defa m - (n-d)$, there are only finitely many homology spheres that cannot be obtained through suspension and one-point suspension from another. Here, $m$ denotes the number of minimal non-faces and $d-1$ is the dimension of $\de$. In Section \ref{sec:kleinalpha} we consider homology spheres with $\al$ up to four. In Section \ref{sec:app} we   classify homology spheres with $n-d\leq 3$, thus generalizing a result of Mani \cite{Mani}. This result is somewhat unrelated to the main topic of this paper but we need it as an auxiliary result.
%
%
%
%
\section{Elementary constructions and notation} \label{sec:elem}
Let \de be a not necessarily pure simplicial complex. A face $F \in \de$ is called a \emph{facet} if it is not contained in another face. 
The \emph{join} of two simplicial complexes $\de$ and $\Gamma$ on disjoint sets of vertices is
\[ \de \join \Gamma \defa \set{S \cup T \with S\in \de, T \in \Gamma} \,.\]
Note that the set of minimal non-faces of $\de \join \Gamma$ is just the union of the sets of minimal non-faces of \de and $\Gamma$. The \emph{link} of a vertex $v$ in a complex $\de$ is $\lk{v} \defa \set{S \in \de \with v \notin S, S \cup \set{v} \in \de}$. For a subset $W \subset [n]$ of the vertices of $\de$ let
$ \de_W \defa \set{S \in \de \with S \subset W}$
denote the \emph{induced subcomplex} of $\de$.
The \emph{one-point suspension} of \de at $v$ is 
\[ \sus \de \defa \set{ S, S\cup \set{v}, S\cup \set{v'} \with v \notin S \in \de} \cup \set{ S \cup \set{v,v'} \with S \in \lk{v}} \,.\]
where $v$ is a vertex of $\de$ and $v'$ denotes a new vertex. Note that every minimal non-face of $\sus \de$ containing $v'$ also contains $v$ and vice versa. We can reconstruct $\de$ from $\sus\de$ as a link: $\de = \lk[\sus\de]{v'}$.
The \emph{two-point suspension} of \de is
\[ \tsus \de \defa \set{ S, S\cup \set{w}, S\cup \set{w'} \with S \in \de } \]
where $w$ and $w'$ are new vertices. Note that the two-point suspension is just the join with the boundary of a $1$-simplex. As above, we can reconstruct $\de$ from $\tsus\de$ as a link: $\de = \lk[\tsus\de]{w}$.
The operations $\sus$ and $\tsus$ commute (if they are defined). See \cite{joswig2005one} for a comprehensive discussion of these constructions. We call a simplicial complex \emph{unsuspended} if it cannot be obtained from another complex by an iterated application of $\sus$ and $\tsus$. 
\begin{lemma} \label{lemma:dprime}
For every simplicial complex $\de$, there is an up to isomorphism unique unsuspended complex $\de'$, such that $\de$ is obtained from $\de'$ by an iterated application of $\sus$ and $\tsus$.
\end{lemma}
\begin{proof}
\newcommand{\mns}{\mathbb M}
Every simplicial complex is uniquely determined by its minimal non-faces. Therefore we can focus on the set $\mns$ of minimal non-faces of \de instead of on \de itself.

Let us describe the effect of the one- and two-point suspensions on $\mns$. A two-point suspension adds a new set $\set{w,w'}$ on two new vertices $w, w'$ to $\mns$, so this new set is disjoint from every other set. Let us call such a set \emph{isolated}. A one-point suspension $\sus$ corresponds to `doubling' the vertex $v$ in $\mns$: The new vertex $v'$ will be added to exactly those sets in $\mns$ that already contain $v$.

If we are given a configuration $\mns$, then it is clear that by removing all isolated sets and all `double' vertices we can construct a minimal configuration $\mns'$ of sets such that $\mns$ can be reproduced using the operations described above. If $v'$ is a `double' of a vertex $v$, i.e. both are contained in exactly the same minimal non-faces, then we have to choose which one to remove. But the two resulting configurations are isomorphic. Therefore, $\mns'$ is unique up to isomorphism. If there is more than one vertex contained in no minimal non-face at all, these vertices are also to be regarded as `double'.

The configuration $\mns$ is the set of minimal non-faces of a simplicial complex if and only if no set in $\mns$ is contained in another one. This property then also holds for $\mns'$, so there is an unique simplicial complex $\de'$ which has the sets in $\mns'$ as its minimal non-faces. This complex is unsuspended, because $\mns'$ was chosen minimal.

There is one degenerate case: If all sets in $\mns$ are isolated and every vertex is contained in a minimal non-face, then let $\de' := \set{\emptyset}$, which is unsuspended. 
\end{proof}
Moreover, if $\de$ is a simplicial complex and $\de'$ is the complex postulated in \ref{lemma:dprime}, then from the definition of $\sus$ and $\tsus$ it follows that $\de'$ is a link of $\de$. Hence $\de$ is a homology sphere, a PL-sphere or polytopal if and only if $\de'$ has the corresponding property. From the above proof we also get a useful criterion for being unsuspended:
\begin{lemma}
A simplicial complex is unsuspended if and only if its minimal non-faces are point-separating, i.e. for every two vertices there is a minimal non-face containing exactly one of them.
\end{lemma}
From this we see that the maximal number of vertices an unsuspended simplicial complex with $m$ minimal non-faces can have is bounded by a function of $m$. Hence there are only finitely many unsuspended complexes with a given number of minimal non-faces. Having this in mind, it seems natural to consider the way the minimal non-faces intersect, as we do with the next definition:
\begin{definition}
Let \de be a $(d-1)$-dimensional simplicial complex on $n$ vertices with $m$ minimal non-faces $M_i ,\, 1\leq i \leq m$. Define $\al \defa m - (n-d)$. Also define the \emph{nerve} of the minimal non-faces of \de to be
\[ \MM \defa \set{S\subset [m] \with \bigcap_{i\in S} M_i \neq \emptyset} \]
\end{definition}
Note that $\MM[\sus\de] = \MM$ and $\MM[\de \join \Gamma]$ is the disjoint union of \MM and $\MM[\Gamma]$, so \de is join-irreducible if \MM is connected. On the other hand, if \MM is not connected \de can be written as a join as follows: The vertices of \de can be partitioned in two non-empty sets $V_1$ and $V_2$ such that no minimal non-face of \de intersects with both sets. Then $\de = \de_{V_1} \join \de_{V_2}$. 
Neither one- nor two-point suspensions alter the value of $\al$ and $\al[\de \join \Gamma] = \al + \al[\Gamma]$. We give a possible interpretation of the invariant $\al$: We call a set of vertices a \emph{blocking set} for the minimal non-faces if it has a non-empty intersection with every minimal non-face. Note that the blocking sets are exactly the complements of the faces of $\de$. Therefore every blocking set has at least $n-d$ vertices. One can always construct a blocking set for the $m$ minimal non-faces using $m$ vertices. The number \al measures how many fewer are needed. From this we immediately get:
\begin{proposition} \label{prop:alphapos}
Every vertex of \de is contained in at most $\al + 1$ minimal non-faces. In particular, $0 \leq \dim{\MM} \leq \al$.
\end{proposition}
Of course, the statement $\al \geq 0$ is just the well known lower bound $m \geq n-d$ on the number of minimal non-faces. Note that $\al = 0$ implies \MM to be zero dimensional, thus all minimal non-faces are disjoint. In this case \de is the join of boundaries of simplices and a (full) simplex.
%
%
%
%
%
\section{Minimal non-faces of homology spheres} \label{sec:lcm}
We will use a result of Gasharov, Peeva and Welker \cite{lcm} which needs some additional notation to be stated. For a set $S \subset [n]$, let $\cc{S} \defa [n]\setminus S$ denote its complement. Fix a field $k$ and a simplicial complex \de on $n$ vertices with minimal non-faces $M_1, \hdots, M_m$. Let 
\[\LL \defa \set{M_{i_1} \cup \hdots \cup M_{i_k} \with \set{i_1,\hdots,i_k}\subset [m]} \]
 denote the lattice of unions of minimal non-faces, ordered by inclusion. Here, the union over the empty set is considered to be the empty set. Furthermore, for $S \in \LL$ let
\[(\emptyset, S)_{\LL} \defa \set{ A \in \LL \with \emptyset \subsetneq A \subsetneq S}\]
denote the open lower interval in $\LL$, thought of as a simplicial complex by considering its order complex; that is the simplicial complex of all linearly ordered subsets. Finally, for a set $S \subset [n]$, we define 
\[ \hh{i}{S} \defa \begin{cases}
\dim \tilde{H}_{i}((\emptyset, S)_{\LL}, k) &\text{if } S \in \LL \\
0 & \text{otherwise}
\end{cases} \]
With this notation, an immediate consequence of Theorem 2.1 in \cite{lcm} is the following:
\begin{proposition} \label{satz:lcm}
Let \de be a $(d-1)$-dimensional homology sphere on $n$ vertices. Then for $S \subset [n]$ and $i \geq 1$ we have:
\[ \hh{i-2}{S} = \hh{n-d-i-2}{\cc{S}} \]
In particular, if $\hh{i-2}{S} > 0$ for any $i$, then $\cc{S} \in \LL$.
\end{proposition}
\begin{proof}
Theorem 2.1 in \cite{lcm} establishes $\beta_{i,S} = \hh{i-2}{S}$ for every simplicial complex, where $\beta_{i,S}$ is the $i$-th multi-graded Betti number of the Stanley-Reisner-Ring of $\de$ over $k$. Every homology sphere is a Gorenstein complex, see \cite[II.5 Theorem 5.1]{stan}, so its minimal free resolution is self dual, see Chapter I.12 in \cite{stan} for details. In particular, we have $\beta_{i,S} = \beta_{n-d-i, \cc{S}}$.
\end{proof}
This result imposes strong conditions on the structure of the minimal non-faces of a homology sphere. It is crucial for most of the sequel. From this we will derive at the end of this section our first main result. But first let us prove the following useful proposition:
\begin{proposition}\label{lemma:eindeutig}
If \de is an unsuspended homology sphere, then \de is uniquely determined by $\MM$. Moreover, the facets of \MM are exactly the sets
\begin{equation}\label{eq:Fi}
 F_i \defa \set{j\in [m]\with i \in M_j} \text{ for } 1 \leq i \leq n \text{ and } F_i \neq F_j \text{ for } i \neq j \,. 
\end{equation}
\end{proposition}
\begin{proof}
We will first prove the second claim. Clearly the facets of \MM are among the $F_i$. To prove the converse, assume there is an $i$ such that $F_i$ is not a facet. Then there is a $j$ with $F_i \subsetneq F_j$, so every minimal non-face containing $i$ also contains $j$ and there exists a minimal non-face $M$ containing $j$ but not $i$. 
Note that $(\emptyset, M)_{\LL} = \emptyset$ and hence $\hh{-1}{M} = 1$, so by \ref{satz:lcm} we have $\cc{M} \in \LL$.
Now $i$ is in $\cc{M}$, so there is a minimal non-face $M^\prime \subset \cc{M}$ containing $i$, which is a contradiction. So, every set $F_i$ is a facet of $\MM$. If $F_i = F_j$ for $i \neq j$, then a minimal non-face contains $i$ if and only if it contains $j$. Hence \de is not unsuspended.

Now let us describe how we can read off $\de$ from $\MM$: Every vertex of \MM corresponds to a minimal non-face of \de and we proved above that every facet of \MM corresponds to a vertex of $\de$. Furthermore, the following can be seen from \eqref{eq:Fi}: A minimal non-face $M_i$ contains a vertex $v_j$ of \de if and only if the vertex $i$ of \MM corresponding to $M_i$ is contained in the facet $F_j$ corresponding to the vertex $v_j$. Thus the minimal non-faces of \de are determined by $\MM$ and hence so is $\de$.
\end{proof}
This correspondence allows us to switch freely between \de and $\MM$. For example, one consequence is that the fact that every minimal non-face contains at least two vertices translates to the statement that every vertex of \MM is contained in at least two facets. Furthermore, every two vertices in $\MM$ lie in a different set of facets, because every two minimal non-faces of \de have different sets of vertices. Another interpretation of \eqref{eq:Fi} is as follows: Consider the vertex-incidence matrix of the minimal non-faces of $\de$, i.e. the matrix $A_1$ with the rows labeled by the vertices of $\de$ and the columns labeled by the minimal non-faces of $\de$. The entry in the $i$th row and $j$th column is $1$ if the $i$th vertex is contained in the $j$th minimal non-face, and zero otherwise. Similarly, consider the vertex-incidence matrix $A_2$ of the facets of $\MM$. Now \eqref{eq:Fi} is equivalent to the statement $A_1 = A_2^t$, where the $t$ means transpose. In the sequel we will sometimes identify the minimal non-faces of \de with the vertices of $\MM$. 

Now we come to the main theorem of this section:
\begin{theorem}\label{satz:alphafinit}
For a given number $\alpha$, there are only finitely many unsuspended homology spheres \de with $\al = \alpha$.
\end{theorem}
The condition that \de is unsuspended is obviously necessary for the theorem to hold, since one can easily construct infinite families of homology spheres with the same \al using the one- and two-point suspensions. However it is not clear if the assumption that \de is a homology sphere is needed. We found the following example of an infinite family of manifolds with boundary which share the same $\alpha$, but the author does not know if the theorem holds for manifolds without boundary.
\begin{example}
Let $\de$ be the complex obtained from the boundary complex of the $k$-th cross-polytope by removing one facet. The resulting complex has $k$ one-dimensional minimal non-faces and one $(k-1)$-dimensional one, the removed facet. For every value of $k$, we have $\alpha(\de) = (k+1) - (2k - k) = 1$. So this is an example of an infinite family of unsuspended simplicial complexes with $\alpha = 1$.
\end{example}
\noindent To prove \ref{satz:alphafinit} we need two lemmata:
\begin{lemma} \label{prop:independent}
Let \de be a simplicial complex and let $G$ denote the graph (i.e. the $1$-skeleton) of $\MM$. Then the maximal size of a set of pairwise non-adjacent edges of $G$ (i.e. the \emph{matching number} of $G$) is at most \al.
\end{lemma}
\begin{proof}
Choose a set of $k$ pairwise non-adjacent edges of $G$. This corresponds to a family of pairs $(M_{1a}, M_{1b}), \hdots , (M_{ka}, M_{kb})$ of different minimal non-faces of $\de$ with $M_{ia} \cap M_{ib} \neq \emptyset$ for every $i$. For every such pair choose a vertex $v_i \in M_{ia} \cap M_{ib}$ of $\de$. Further, for every minimal non-face $M$ of \de which does not appear in the above list choose an additional vertex $v \in M$. This way we get a collection of $k + (m-2k) = m-k$ vertices of \de which has a non-empty intersection with every minimal non-face, thus it is a blocking set. This yields $d \geq n - (m-k)$ and therefore $k \leq m-n+d = \al$. 
\end{proof}
\begin{lemma} \label{lemma:alphaschneiden}
Let \de be homology sphere. Every minimal non-face of \de intersects at most \al other minimal non-faces nontrivially. Equivalently, the degree of every vertex of the graph of \MM is at most $\al$.
\end{lemma}
\begin{proof}
The two formulations are clearly equivalent, since two vertices of \MM are connected by an edge if and only if the corresponding minimal non-faces intersect nontrivially.
Let $M$ be a minimal non-face of $\de$. Let $\mathcal O$ denote the order complex of $(\emptyset, \cc{M})_{\LL}$. As pointed out above we have $1 = \hh{-1}{M} = \hh{n-d-3}{\cc{M}}$ by \ref{satz:lcm}. Hence $\mathcal O$ has a nontrivial $(n-d-3)$-th homology. This implies that $\mathcal O$ has a $(n-d-3)$-dimensional face and hence $(\emptyset, \cc{M})_{\LL}$ contains an element of rank $n-d-2$. But since $\mathcal O$ is not a cone there cannot be only one element of a given rank, so there are at least two elements of rank $n-d-2$. These are different unions of (at least) $n-d-2$ minimal non-faces each, hence $\cc{M}$ contains at least $n-d-1$ minimal non-faces. So the maximal number of minimal non-faces intersecting with $M$ is $m - 1 - (n-d-1) = \al$.
\end{proof}
This lemma fails for general simplicial complexes. Since this is crucial for the extent to which \ref{satz:alphafinit} holds, we give a second proof, which is somewhat simpler but involves more algebra:
\begin{proof}[Second proof of \ref{lemma:alphaschneiden}]
Again let $M$ be a minimal non-face of $\de$. Consider the minimal free resolution of the Stanley-Reisner-Ring of $\de$. As above we have $1 = \hh{-1}{M} = \hh{n-d-3}{\cc{M}} = \beta_{n-d-1,\cc{M}}$ by \ref{satz:lcm}. So there is a component of the $(n-d-1)$-th syzygy whose grading is shifted by $\cc{M}$. In the resolution its generator cannot be mapped to zero because the resolution is minimal. Therefore it is mapped to a component of the $(n-d-2)$-th syzygy whose grading is shifted by a subset of $\cc{M}$. This subset is proper since the resolution is minimal and it is again a union of minimal non-faces. Going down the resolution we get a chain of subsets $\cc{M} = S_{n-d-1} \supsetneq S_{n-d-2} \supsetneq \hdots \supsetneq S_{1}$ with $\beta_{i, S_i} > 0$ for every $i$, so each set $S_i$ is an union of minimal non-faces. Therefore, $\cc{M}$ contains at least $n-d-1$ minimal non-faces. Now the claim follows as above.
\end{proof} 
\begin{proof}[Proof of \ref{satz:alphafinit}]
Let $\de$ be an unsuspended homology sphere with $\alpha = \al$. Choose a maximal family of pairs $(M_{1a}, M_{1b}), \hdots , (M_{ka}, M_{kb})$ of different minimal non-faces of $\de$ with $M_{ia} \cap M_{ib} \neq \emptyset$ for every $i$. By \ref{prop:independent} we know $k \leq \al$. Since $\de$ is unsuspended $\MM$ has no isolated vertices. Hence every minimal non-face $M$ of $\de$ which is not contained in the above list has a non-empty intersection with at least one minimal non-face in that family. Otherwise we could choose a further minimal non-face $M'$ with $M\cap M' \neq \emptyset$ which contradicts the maximality of the family. Furthermore by \ref{lemma:alphaschneiden} every minimal non-face of the family intersects at most $\al$ other minimal non-faces non-trivially. So the number of minimal non-faces of \de is bounded by $2\alpha + (2\alpha)^{\alpha}$. As noted above, there are only finitely many simplicial complexes with at most that number of minimal non-faces.
\end{proof}

%
%
%
%
\section[Small values of Alpha]{Small values of $\alpha$} \label{sec:kleinalpha}
In this section, we determine the homology spheres with $\al \leq 4$ by examining the complex $\MM$. Since the effect of the one- and two-point suspension are well understood, we restrict ourselves to unsuspended homology spheres. If $P$ is a simplicial polytope then let $\bd{P}$ denote its boundary complex. For a given number $d$, we will denote with $P_d$ the polytope obtained from the $d$th cross-polytope by adding a pyramid over one of its facets. Note that $\al[\bd{P_d}] = d$. For example, $P_2$ is a pentagon.
The main result of this section is:
\begin{theorem} \label{satz:alpha0123}
There are no unsuspended homology spheres with $\alpha=0$ or $\alpha = 1$. For $\alpha = 2$ resp. $\alpha = 3$, the only unsuspended homology spheres are $\bd{P_2}$ resp. $\bd{P_3}$. For $\alpha = 4$, there are exactly five unsuspended homology spheres: $\bd{P_2} \join \bd{P_2}$, $\bd{P_4}$, the boundary complex of a cyclic $4$-polytope with seven vertices $C_4(7)$ and the boundary complexes of two additional $4$-polytopes $Q_1$ and $Q_2$, see Figure \ref{fig:Q1Q2}.
In particular, all homology spheres with $\al \leq 4$ are polytopal.
\end{theorem}
\begin{figure}
\includegraphics[scale=0.8]{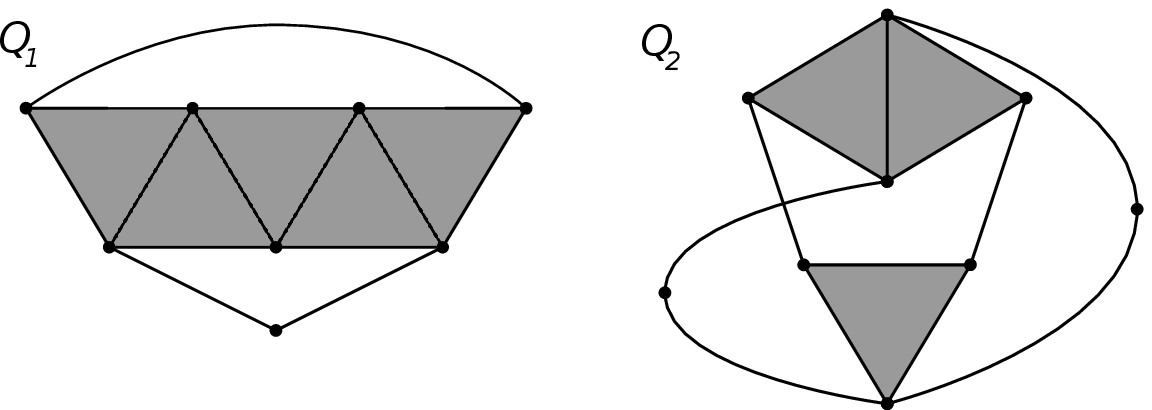}
\caption{The complexes $\MM[\bd{Q_1}]$ and $\MM[\bd{Q_2}]$}
\label{fig:Q1Q2}
\end{figure}
As a corollary, we get an enumeration of the unsuspended homology spheres with a small absolute number of minimal non-faces:
\begin{corollary}\label{cor:8mns}
Every join-irreducible homology sphere with at most eight minimal non-faces is one of the following: The boundary of a simplex, the boundary of a pentagon, $\bd{C_4(7)}$, $\bd{Q_1}$ or an (possibly iterated) one-point suspension of one of them. In particular all these homology spheres are polytopal.
\end{corollary}
\begin{proof}
For an unsuspended homology sphere with $m \leq 8$ minimal non-faces either $n-d \leq 3$ or $\al = m - (n-d) \leq 4$. The case $n-d \leq 3$ will be proved in the Appendix, see Section \ref{sec:app}.
\end{proof}
\begin{Remark}
\begin{enumerate}
	\item Using a computer we found $21$ different unsuspended homology spheres with $\al = 5$ and $13$ different unsuspended homology spheres with nine minimal non-faces. There might be more in both cases.
	\item For $n-d = 4$, part of the above statement is that there are no homology spheres with $5$ minimal non-faces. This is a special case of  \cite[Theorem 3.2]{hibig}, where it is shown that a Gorenstein monomial ideal of projective dimension $3$ cannot be minimally generated by $5$ monomials.   
	\item Note that for \al up to four, the dimension of \de is always equal to $\al - 1$, but this is no longer true for $\al = 5$. It is obvious that the dimension can be smaller than $\al -1$, since there are many $2$- and $3$-polytopes other than $P_2$ and $P_3$. It is more difficult to see that the dimension can also be larger than $\al-1$. Kenji Kashiwabara constructed a five dimensional unsuspended homology sphere with $\al = 5$, see Figure \ref{fig:Kenjis}. %
	We found a polytopal realization of it. 
	\item Further properties all unsuspended homology spheres with $\al \leq 4$ share are that the diameter of the graph of $\MM$ is exactly two and that $\MM$ is isomorphic to the complex generated by the set of minimal non-faces. However, for both properties there are counterexamples with $\al = 5$.
	\item As mentioned above, all homology spheres with $\al \leq 4$ are polytopal. It might be interesting to ask for the smallest value of $\al$ for which there is a non-polytopal homology sphere. The smallest example of a non-polytopal sphere we know is the Barnette sphere \cite{Barnette197337} with $\al = 9$. Also, one might consider homology spheres which are not topological spheres. The examples we found have an $\alpha$ around 300.
\end{enumerate}
\end{Remark}
\begin{figure}
\includegraphics[scale=1.2]{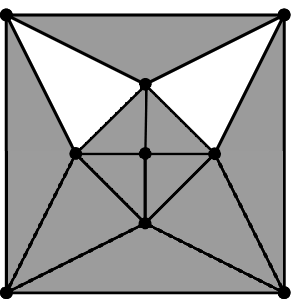}
\caption{The \MM of Kashiwabara's example}
\label{fig:Kenjis}
\end{figure}
To prove \ref{satz:alpha0123}, we need some preparations. 
Let us describe a certain way to use \ref{satz:lcm}, which we found especially useful:
Let \de be an unsuspended homology sphere.
Consider the situation where $V$ is a set of vertices $v_1,\hdots,v_l$ of $\MM$ and $F$ is a facet of $\MM$, such that $V \cap F = \emptyset$ and every vertex of $F$ shares an edge with some vertex of $V$.
We will denote this situation as \I{V}{F}.
Let $S \in \LL$ denote the union of the minimal non-faces $M_i$ of \de corresponding to the elements $v_i $ of $V$. 
The vertex $v_F$ of $\de$ corresponding to $F$ lies in $\cc{S}$, but every minimal non-face containing $v_F$ intersects a minimal non-face contained in $S$ non-trivially.
In this situation, $\cc{S} \notin \LL$, so $\hh{i}{S} = 0$ for every $i$. In particular, either one of the minimal non-faces $M_i$ is contained in the union the other minimal non-faces, or there is an additional minimal non-face contained in $S$. This translates to the following statement about $\MM$: In the first case, one of the vertices of $V$, say $v_1$, has the property that every facet containing $v_1$ contains also another vertex in $V$. In the second case, there is an vertex $w \notin V$ of $\MM$, such that every facet containing $w$ contains also a vertex in $V$.
In the sequel, we will identify the vertices of $\MM$ with the minimal non-faces of $\de$.
\begin{lemma}\label{lemma:ksimplex}
Let \de be an join-irreducible unsuspended homology sphere. Assume that the complex \MM contains a $k$-simplex $T$ and every vertex of $T$ is connected to the rest of the complex only via one edge. Then $\de = \bd{P_k}$.
\end{lemma}
\begin{proof}
For $0 \leq i \leq k$ let $A_i$ denote the vertices of the $k$-simplex and let $B_i$ be the vertex connected to $A_i$ via the edge. We will think about the $A_i$ and $B_i$ as vertices of $\MM$ and at the same time as minimal non-faces of $\de$. It is $B_i \neq B_j$ for $i \neq j$, because otherwise we get a contradiction with \I{A_i}{ \set{A_j,B_j}}. Similarly, $B_i \cap B_j = \emptyset$ for $i\neq j$ because of \I{\set{A_i,B_i}}{\set{A_j,B_j}}. By \I{\set{B_1, \hdots, B_k}}{\set{A_1, \hdots, A_k}} we know that there is an additional minimal non-face $C \subset B_1 \cup \hdots \cup B_k$. $C$ has a non-empty intersection with at least one of the $B_i$, say $B_0$. For every $i \neq 0$ it follows from \I{\set{A_i, C}}{\set{A_0, B_0}} that there is a minimal non-face $D_i \subset A_i \cup C$. Since $D_i$ intersects $A_i$ and $C$ nontrivially, it follows $D_i = B_i$. The case $i=0$ is analogous, once we know there is a second $B_j$ intersecting with $C$. So for every $i$, $B_i \subset A_i \cup C$. It remains to prove that there is no further minimal non-face, because then \MM and thus \de are uniquely determined. Assume, there is an additional minimal non-face $E$. Since \MM is connected and $C \subset B_1 \cup \hdots \cup B_k$, we may assume $E \cap B_i \neq \emptyset$ for an $i$. For $j \neq i$ follows from \I{\set{E, A_j}}{\set{A_i, B_i}} that there is a minimal non-face $F \subset E \cup A_j$. It is $F = B_j$, because otherwise $E$ would be one of the $B_i$. Hence $B_j \subset E \cup A_j$ for every $j$. But from this we get $C \subset B_0\setminus A_0 \cup \hdots \cup B_k\setminus A_k \subset E$, a contradiction.
\end{proof}
\begin{figure} 
\includegraphics{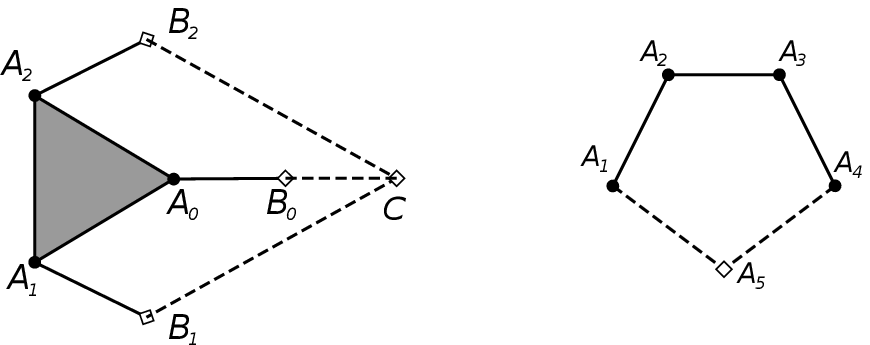}
\caption{The situation of \ref{lemma:ksimplex} for $k=2$ and the situation of \ref{prop:dimmm}}
\label{fig:ksimplex}
\end{figure}
\noindent Using this result we will now characterize the cases where the dimension of \MM is extremal:
\begin{proposition}\label{prop:dimmm}
Let \de be an unsuspended homology sphere. If \MM is one-dimensional, then every join-component of $\de$ is $\bd{P_2}$. 
\end{proposition}
\begin{proof}
It suffices to consider the case where $\de$ is join-irreducible. 
Choose an edge of $\MM$. It corresponds to two different minimal non-faces $A_1, A_2$ of $\de$ with $A_1 \cap A_2 \neq \emptyset$. Since there is a vertex $v \in A_2 \setminus A_1 \subset \cc{A_1}$ and $\cc{A_1} \in \LL$, there is a third minimal non-face $A_3 \subset \cc{A_1}$ with $v \in A_2 \cap A_3 \neq \emptyset$. Using the same argument we know that there is a fourth minimal non-face $A_4$ with $A_3 \cap A_4 \neq \emptyset$. Moreover, we have $A_1 \neq A_4$, because $A_1 \cap A_3 = \emptyset$. Since $\dim{\MM} = 1$, $A_2 \cap A_3$ is contained only in $A_2$ and $A_3$. Therefore it follows from \I{\set{A_1, A_4}}{ \set{A_2, A_3}} that there is a fifth minimal non-face $A_5 \subset A_1 \cup A_4$. If $A_5$ would intersect with an additional minimal non-face $B$, then $A_5 \cap B$ would not be contained in $A_1 \cup A_4$, so $A_5$ intersects only with these two nontrivially. In a similar way it follows from \I{\set{A_3, A_5}}{\set{A_1, A_2}} that $A_4 \subset A_3 \cup A_5$ and thus intersects only with these two. Now the claim follows from \ref{lemma:ksimplex} with the $1$-simplex $\set{A_4, A_5}$.
\end{proof}

\begin{proposition}\label{prop:dimmm2}
Let \de be a homology sphere with $\al > 0$. Then $\dim \MM \leq \al - 1$. If \de is unsuspended and $\dim \MM = \al - 1$, then $\de = \bd{P_{\al}}$
\end{proposition}
\begin{proof}
Both the dimension of $\MM$ and \al are invariant under one- and two-point suspensions. Therefore it suffices to consider unsuspended homology spheres.
For general simplicial complexes we already proved $\dim \MM \leq \al$ in \ref{prop:alphapos}. Now if \de is an unsuspended homology sphere this bound can be improved as follows: Assume $\dim{\MM} = \al$, then $\MM$ has a facet $F$ with $\al+1$ vertices. Every one of these vertices shares an edge with every other vertex, so their degree is $\al$, the maximum possible value according to \ref{lemma:alphaschneiden}. But this implies that no vertex of $F$ shares an edge with a vertex not in $F$, so $F$ is isolated. Hence every vertex of $F$ is contained in only one facet of $\MM$. But this contradicts the remark below \ref{lemma:eindeutig}, because it would imply that the corresponding minimal non-faces have only one vertex.

For the second part, assume $\dim \MM = \al - 1$. We start with proving that $\de$ is join-irreducible. 
Assume $\de = \de_1 \join \de_2$ is not join-irreducible. We have $\al[\de_i] > 0$ for $i=1,2$, because $\al[\de_i] = 0$ implies that $\de_i$ is a join of boundaries of simplices, see the remark below \ref{prop:alphapos}. In this case $\de$ would fail to be unsuspended. Moreover, $\MM[\de] = \MM[\de_1] \cup \MM[\de_2]$, so $\dim \MM[\de] = \max \set{\dim \MM[\de_1], \dim\MM[\de_2]}$. Hence we may assume without loss of generality that $\dim \MM[\de_1] = \dim \MM[\de]$. But $\al[\de_1] = \al - \al[\de_2] < \al$, so we get $\dim \MM[\de_1] = \al - 1 > \al[\de_1] - 1$, a contradiction. Thus we conclude that $\de$ is join-irreducible.

Next choose an $(\al - 1)$-simplex $F$ in $\MM$ with vertices $v_1, \hdots, v_{\al}$. Every two vertices of $\MM$ lie in a different set of facets. Hence every $v_i$ is contained in at least one further facet. These facets contain vertices not in $F$. Therefore, there are vertices $w_1, \hdots, w_{\al}$ such that $v_i$ and $w_i$ share an edge. But now the degree of every $v_i$ is $\al - 1 + 1 = \al$, so $v_i$ cannot share an edge with an additional vertex. It suffices to prove that these edges are facets of $\MM$, because then we are in the situation of \ref{lemma:ksimplex}. 

To prove this assume that there is a $v_i$ such that the edge $\set{v_i, w_i}$ is not a facet but contained in some larger facet $S$ of $\MM$. The vertices of $S$ are a subset of the vertices of $F$ and $w_i$, because these are all vertices connected to $v_i$ by an edge. Let $v_j$ be another common vertex of $F$ and $S$. Since $v_i$ and $v_j$ lie in different sets of facets, there is another facet $S'$ containing $v_i$ and $v_j \notin S'$. Again, the vertices of $S'$ are a subset of the vertices of $F$ and $w_i$. But now we get a contradiction with $\I{v_j}{S'}$: The vertex $v_{S'}$ of $\de$ corresponding to $S'$ lies in the complement of the minimal non-face $M_j$ corresponding to $v_j$, but every minimal non-face containing $v_{S'}$ intersects $M_j$ nontrivially.
\end{proof}
\noindent We are now ready to prove \ref{satz:alpha0123} for $\al \leq 3$. The case $\al = 4$ is very technical, so we will only sketch the proof.
\begin{proof}[Proof of \ref{satz:alpha0123}]
If $\al = 0$, then \de is a join of boundary of simplices and thus never unsuspended. If $\al = 1$, then $\dim{\MM} \leq 0$ and \de is again a join of boundary of simplices, so it follows $\al = 0$, a contradiction. In the cases $\al = 2,3$, we may assume $\de$ to be join-irreducible, because \al is join-additive.  Now the claim follows from \ref{prop:dimmm} and \ref{prop:dimmm2}.

Now consider the case $\al = 4$. The dimension of \MM can be two or three. If it is three, then $\de = \bd{P_4}$, so we only need to consider the case $\dim\MM = 2$. If $\de$ can be written as a join then $\de = \bd{P_2}\join\bd{P_2}$, since $\al$ is join-additive. Therefore it suffices to consider join-irreducible complexes. Since $\dim\MM = 2$, we know that $\MM$ contains a $2$-simplex. One of its vertices is contained in four edges, because otherwise we are in the situation of \ref{lemma:ksimplex}. These edges can be part of other $2$-simplices. 
Now carefully considering all possible configurations of $2$-simplices and edges around the fixed vertex yields the result. Polytopal realizations of $Q_1$ and $Q_2$ were found computationally.
\end{proof}
%
%
%
%
%
\section{Appendix: Small codimension}\label{sec:app}
In this section we classify homology spheres with $n-d \leq 3$, thus completing the proof of \ref{cor:8mns}. For a start, if \de is a homology sphere with $n-d=1$, then \de is easily seen to be the boundary of a simplex. The next case needs some further reasoning:
\begin{proposition}
Let \de be a homology sphere with $n-d = 2$. Then \de is the join of the boundaries of two simplices. In particular, \de is not unsuspended.
\end{proposition}
\begin{proof}
Let $M$ be a minimal non-face of $\de$. Then by \ref{satz:lcm} we have $1 = \hh{-1}{M} = \hh{2-3}{\cc{M}}$. This implies $(\emptyset, \cc{M})_{\LL} = \emptyset$, so $\cc{M}$ is itself a minimal non-face. Let $\de'$ be the simplicial complex with the minimal non-faces $M$ and $\cc{M}$. This is the join of the boundaries of two simplices and hence a $(d-1)$-homology sphere. But $\de \subset \de'$ and both are homology spheres of the same dimension, so we have $\de = \de'$.
\end{proof}
The final case $n-d=3$ is more involved. It could be proved using the methods developed in this paper, but we derive it from an algebraic result in \cite{Kamoi}, since this yields a shorter proof. Note that this generalizes the well-known result of Mani \cite{Mani}, which states that every sphere with $n-d=3$ is polytopal.
\begin{proposition}
Let \de be an unsuspended homology sphere with $n-d = 3$. Then $m$ is odd, $n=m \geq 5$ and \de is the boundary complex of the cyclic polytope $C_{d}(n)$.
\end{proposition}
\begin{proof}
 The Stanley-Reisner-Ring of \de is known to be Gorenstein and the number $n-d$ is its codimension, so we are in the situation of Theorem 0.1 in \cite{Kamoi}. Let $s \defa (m+1)/2$ and $\nu: \NN \rightarrow [m]$ be map sending $i$ to $i$ modulo $m$, as in \cite{Kamoi}. The third equivalent condition in that Theorem states that $m$ is odd and there are pairwise coprime monomials $b_1, \hdots, b_m$ such that the Stanley-Reisner-Ideal is generated by the monomials $\prod_{j=1}^{s-1} b_{\nu(j+i)}$ for $1 \leq i \leq m$. These generators correspond to the minimal non-faces of $\de$, so if $B_i$ is the set of variables in $b_i$, the minimal non-faces of \de are
\[ M_i \defa \bigcup_{j=1}^{s-1} B_{\nu(j+i)}  \text{ for } 1 \leq i \leq m\]
The sets $B_i$ are pairwise disjoint, because the $b_i$ are pairwise coprime. So if $m=3$, the minimal non-faces are disjoint and thus \de is a join of the boundaries of three simplices. In this case \de fails to be unsuspended, so we have $m \geq 5$. 
Since \de is unsuspended, its minimal non-faces are point-separating. Hence every set $B_i$ contains exactly one element. Also, every vertex of \de is contained in at least one minimal non-face, because otherwise \de would be a cone. This implies $m=n$ and the minimal non-faces are now completely determined and so is $\de$. 

To prove that \de is indeed the boundary of $C_{d}(n)$ we use Gale's Evenness Condition, see \cite[Theorem 0.7]{zie}. We reorder the vertices of \de by assigning the position $\nu(2i)$ to the vertex in $B_i$. If $2(i+1) < m$ then the minimal non-face $M_i$ contains the odd vertices $v_j$ with $j \leq 2i-1$ and the even vertices with $j \geq 2(i+1)$. Moreover, if $2(i+1) > m$ then $M_i$ contains the odd vertices $v_j$ with $j \geq 2(i+1)-m$ and the even vertices with $j \leq 2i-1-m$. A facet $F$ of \de is a set of $n-3$ vertices which does not contain a minimal non-face. The later condition is equivalent to the statement that the three vertices not in $F$ form the pattern even-odd-even or odd-even-odd. This implies Gale's Evenness Condition.
\end{proof}
From these results, it is clear that the only unsuspended homology spheres with $n-d\leq 3$ and no more than eight minimal non-faces are $\bd C_2(5)$ (the pentagon) and $\bd C_4(7)$.
\section*{Acknowledgments}
I would like to thank Vic Reiner and Volkmar Welker for inspiring hints and discussions. Also I would like to thank Kenji Kashiwabara for constructing the example in Figure \ref{fig:Kenjis}.

\bibliography{Fewmnf}
\bibliographystyle{amsplain}

\end{document}